\documentclass[a4paper]{amsart}
\usepackage{amscd,amssymb}

\usepackage[all]{xy}

\renewcommand{\mod}{\operatorname{mod}\nolimits}

\newcommand{\add}{\operatorname{add}\nolimits}
\newcommand{\repdim}{\operatorname{repdim}\nolimits}

\newcommand{\Hom}{\operatorname{Hom}\nolimits}
\newcommand{\End}{\operatorname{End}\nolimits}

\renewcommand{\Im}{\operatorname{Im}\nolimits}
\newcommand{\Ker}{\operatorname{Ker}\nolimits}

\newcommand{\rrad}{\mathfrak{r}}
\newcommand{\rad}{\operatorname{rad}\nolimits}

\newcommand{\Soc}{\operatorname{Soc}\nolimits}

\newcommand{\gldim}{\operatorname{gldim}\nolimits}

\newcommand{\DTr}{\operatorname{DTr}\nolimits}

\newcommand{\Ext}{\operatorname{Ext}\nolimits}

\newcommand{\op}{{\operatorname{op}\nolimits}}

\newcommand{\pd}{{\operatorname{pd}\nolimits}}

\newcommand{\comp}{\operatorname{\scriptstyle\circ}}

\renewcommand{\L}{\Lambda}

\renewcommand{\P}{{\mathcal P}}

\newcommand{\extto}{\xrightarrow}

\newtheorem{lem}{Lemma}[section]
\newtheorem{prop}[lem]{Proposition}
\newtheorem{cor}[lem]{Corollary}
\newtheorem{thm}[lem]{Theorem}
\theoremstyle{definition}

\newtheorem{example}[lem]{Example}

\begin{document}

\title{Mutation of Auslander generators}
\author[Lada]{Magdalini Lada}
\thanks{The author was partially supported by NFR Storforsk grand No. 167130.}
\address{Magdalini Lada\\Institutt for matematiske fag\\
NTNU\\ N--7491 Trondheim\\ Norway} \email{magdalin@math.ntnu.no}

\date{\today}
\maketitle

\begin{abstract}
Let $\L$ be an artin algebra with representation dimension equal to
three and $M$ an Auslander generator of $\L$.  We show how, under
certain assumptions, we can mutate $M$ to get a new Auslander
generator whose endomorphism ring is derived equivalent to the
endomorphism ring of $M$. We apply our results to selfinjective
algebras with radical cube zero of infinite representation type, where
we construct an infinite set of Auslander generators. 
\end{abstract}

\section*{Introduction}
The notion of the representation dimension for an artin algebra $\L$, was introduced by Auslander in ~\cite{A}. The purpose was to give a measure of how far  $\L$ is from being representation finite (that is, having a finite number of nonisomorphic indecomposable finitely generated modules). It is defined as the smallest number, that can be realized as the global dimension of the endomorphism ring of a finitely generated $\L$-module, which contains all the indecomposable projective and indecomposable injective $\L$-modules as direct summands. Such a module, whose endomorphism ring has global dimension equal to the representation dimension of $\L$, is called an \emph{Auslander generator} of $\L$. 

Although a lot of research focuses on determining bounds and computing the representation dimension of special classes of artin algebras, not much work has been done on studying the class of Auslander generators of an algebra. In this paper, in an effort to understand better the structure - if any - of the class of Auslander generators, we show how we can construct new Auslander generators from given ones, under certain assumptions, for artin algebras with representation dimension equal to three. This is done by replacing an indecomposable direct summand of an Auslander generator by another indecomposable $\L$-module. We call this process mutation since it resembles the mutation of objects in  ~\cite{BMRRT}, ~\cite{Iy}, ~\cite{GLS}, ~\cite{IY}.

For the presentation of our results, we have found it convenient to use the language of \emph{relative homological algebra}, as this was developed by Auslander and Solberg in ~\cite{AS}. In Section 1, we give definitions of all the notions, from relative homological algebra, that appear in our work and we refer to ~\cite{AS} for details. A short overview can also be found in ~\cite{L1}. In Section 2, we fix the setting and make some assumptions under which we are able to mutate a given Auslander generator in order to get a new one. In Section 3, we show that these assumptions are satisfied for infinite type selfinjective algebras with radical cube zero, and a specific class of Auslander generators of these algebras. Finally, in Section 4, we show via an example how we can construct by mutation an infinite set of Auslander generators of an infinite type selfinjective algebra with radical cube zero, starting from a \emph{canonical} one. It is proven in ~\cite{L2} that for the exterior algebra in two variables, this process gives a complete set of \emph{minimal} Auslander generators. 

\section{Preliminaries}

The purpose of this section is to give, in a compact form, some necessary background for understanding the results  in this paper. 

We begin by fixing some notation that we use in the rest of the paper. Let $\L$ be an artin algebra. We denote by $\mod\L$ the category of finitely generated left $\L$-modules. By a $\L$-module we always mean a module in $\mod\L$. For a $\L$-module $M$, we denote by $\add M$ the full subcategory of $\mod\L$ consisting of all direct summands of a direct sum of copies of $M$. A \emph{generator-cogenerator} for $\mod \L$, is a $\L$-module $M$, such that all the indecomposable projective and the indecomposable injective $\L$-modules are in $\add M$. A formal definition of the representation dimension of $\L$, which we denote by $\repdim\L$ is given as follows:
\[\repdim\L=\inf \{\gldim\End_\L(M) \mid M \ \ \rm{generator-cogenerator \  for}  \ \mod\L\}.\]
Next, we recall some basic terminology and results from relative homological algebra and we refer to ~\cite{AS} for more details. Let $M$ be a generator-cogenerator for $\mod\L$. For any pair of $\L$-modules $A$ and $C$, we define the following set of short exact sequences
\begin{multline}F_M(C,A):=\{0\to A\to B \to C\to 0 \mid \Hom_\L(M,B)\to \notag\\\Hom_\L(M,C)\to 0 \mathrm{\
is \ exact}\}.
\end{multline}
A short exact sequence $ (\eta)\colon 0\to A \to B \to C \to 0$, in $\mod\L$, is called \emph{$F_M$-exact} if $(\eta)$ is in $F_M(C,A)$.
It is proven in ~\cite{AS}, that the above assignment defines an additive sub-bifunctor of the bifunctor $\Ext_\L^1(-,-)\colon (\mod\L)^\op \times \mod\L\to \mathrm{Ab}$, where $\mathrm{Ab}$ denotes the category of abelian groups. Moreover, using this fact, it is proven that the set of all $F_M$-exact sequences is closed under pushouts, pullbacks and direct sums of short exact sequences. In other words, any pushout or pullback of an $F_M$-exact sequence is again $F_M$-exact, and the direct sum of any two $F_M$-exact equences is again $F_M$-exact. It is obvious, by the definition of $F_M$-exact sequences, that the $\L$-modules in $\add M$ play the role of projective modules for the set of $F_M$-exact sequences. Hence it is natural to define as \emph{$F_M$-projective} modules, the modules in $\add M$. Dually, for any  $\L$-module $C$, there exists an $F_M$-exact sequence $0\to K \to P\to C\to 0$, with $P$ an $F_M$-projective $\L$-module. Note that the map $P\to C$ is nothing but a right $\add M$-approximation of $C$ which we know is an epimorphism, since $\add M$ contains all the indecomposable projective $\L$-modules. An \emph{$F_M$-projective resolution} of $C$ is an exact
sequence
\[ \cdots\to P_l\extto{f_l} P_{l-1}\to \cdots \to P_1\extto{f_1} P_0 \extto{f_0} C \to 0,\]
where the $\L$-module $P_i$ is $F_M$-projective  for all $i$, and each short exact sequence $0\to\Im f_{i+1}\to P_i \to \Im
f_i\to 0$ is $F_M$-exact. 

Similarly to the ordinary syzygy functor $\Omega_\L\colon \underline{\mod}\L\to \underline{\mod}\L$, we can define the \emph{$F_M$-syzygy} functor  $\Omega_{F_M}\colon \underline{\mod}_{F_M}\L\to\underline{\mod}_{F_M}\L$. Here, $\underline{\mod}_{F_M}\L$ denotes the factor category $\mod\L/\P_{F_M}$, where, for any $\L$-modules $A$ and $B$, the ideal $\P_{F_M}(A,B)\subseteq \Hom_\L(A,B)$ consists of the morphisms that factor through an $F_M$-projective module.  We denote by $\underline{\Hom}_{F_M}(A,B)$ the factor $\Hom_\L(A,B)/\P_{F_M}(A,B)$.
 
 The
\emph{$F_M$-projective dimension} of $C$, which we denote by $
\pd_{F_M} C $, is defined to be the
smallest $n$ such that there exists an $F_M$-projective resolution
\[0\to P_n\to\cdots\to P_1\to P_0 \to C\to 0.\] 
If such $n$ does not exist we set
$\pd_F C=\infty$. Then, the \emph{$F_M$-global dimension} of $\L$ is defined as:
\[\gldim_{F_M} \L =\sup\{\pd_{F_M} C  \mid C\in\mod\L\}.\]
Dually,  if $M$ is a generator-cogenerator for $\mod\L$, we can define a set of exact sequences, which we call \emph{$F^M$-exact}, such that  the $\L$-modules in $\add M$ play the role of injective modules.
The following useful result was proved in ~\cite{L1}.
\begin{prop}\label{gldim}
Let $M$ be a generator-cogenerator for $\mod\L$. Then, for any positive integer $l$, the
following are equivalent:
\begin{enumerate}
\item[(a)] $\gldim\End_{\L}(M)\leq l+2$,
\item[(b)] $\gldim_{F_{M}}\L\leq l$,
\item[(c)] $\gldim_{F^{M}}\L\leq l$.
\end{enumerate}
\end{prop}

In view of Proposition \ref{gldim} we can give the following equivalent definition of the representation dimension of an artin agebra $\L$:
\[\repdim\L=\inf \{\gldim_{F_M}\L\mid M \ \ \rm{generator-cogenerator \  for}  \ \mod\L\}+2.\]
If $M$ is a basic Auslander generator of a selfinjective artin algebra $\L$, we write $M=\L\oplus M_\P$, where $M_\P$ is the direct sum of the non-projective direct summands of $M$. We have the following consequence of Proposition \ref{gldim}.

\begin{cor}\label{DTr}
Let $\L$ be a selfinjective artin algebra. Then the $\L$-module $\L\oplus M_\P$ is an Auslander generator if and only if the $\L$-module $\L\oplus \DTr M_\P$ is an Auslander generator.
\end{cor}

\begin{proof}
The statement is a straightforward consequence of Proposition \ref{gldim} and the fact that $F_M=F^{\DTr M}$ (see ~\cite{AS}). 
\end{proof}

%For an artin algebra $\L$, a \emph{node} is a non-projective, non-injective simple $\L$-module $S$, such that the middle term of the almost split sequence starting at $S$ is projective. If $\L$ is in addition selfinjective, then $\L$ has no nodes if and only if $\L$ has no  indecomposable projective modules of Loewy length two. If $\L$ is a selfinjective algebra without nodes, Corollary \ref{DTr} can also be obtained as an application of ~\cite[Corollary 3.5]{D}. Another application of ~\cite[Corollary 3.5]{D} is the following. 

%\begin{cor}\label{Omega}
%Let $\L$ be a selfinjective algebra without nodes. Then the $\L$-module $\L\oplus M_\P$ is an Auslander generator if and only if the $\L$-module $\L\oplus \Omega_\L(M_\P)$ is an Auslander generator.
%\end{cor}

\section{Mutation of Auslander generators}

In this section, we describe how we can construct, under certain assumptions, an Auslander generator from a given one. In particular, we give  some sufficient setting and assumptions under which we can replace an indecomposable direct summand of an Auslander generator, by another  indecomposable module, so that we get again an Auslander generator.

Let $\L$ be an artin algebra of infinite representation type and $M$ a
generator-cogenerator for $\mod\L$ such that $\gldim_{F_M}\L=1$. Let $N$ be an
indecomposable direct summand of $M$ which is not projective or injective and
\[0\to N \extto{\zeta} E \extto{\varepsilon} \tau^{-1}N \to 0\]
the almost split sequence starting at $N$. Set $M^{\star}=M/N \oplus \tau^{-1}N$, where $\tau$ is the Auslander-Reiten translation. We assume that $E$ is in $\add M$. Note here that the factor module $M/N$ is the cokernel of the split monomorphism $N\hookrightarrow M$. Our goal is to give necessary and sufficient conditions such that $\gldim_{F_{M^\star}}\L=1$. To do this we first need a couple of technical lemmas. With the above assumptions and notation we have the following lemma.

\begin{lem}\label{syzygies}
Let $C$ be an indecomposable $\L$-module. Denote by $l$ the length of the top of
$\underline{\Hom}_{F_{M/N}}(\tau^{-1}N,C)$ as an
$\underline{\End}_{F_{M/N}}(\tau^{-1}N)$-module. Let also $\alpha_C\colon M_C \to C$ be a
minimal right $\add M/N$-approximation of $C$. Then

\begin{enumerate}
\item[(a)]$\Omega_{F_{M/N}}(C)\simeq N^l\oplus \ker\alpha_C/N^l$
\item[(b)]$\Omega_{F_{M^{\star}}}(C)\simeq E^l\oplus \ker\alpha_C/N^l$
\end{enumerate}
and $\ker\alpha_C/N^l$ does not contain $N$ as a direct summand.
\end{lem}

\begin{proof}
Assume first that $l=0$. Then all morphisms from $\tau^{-1}N$ to $C$, factor through
$\alpha_C$. This means that $\alpha_C$ is also a minimal right $\add
M^{\star}$-approximation of $C$. Hence, in this case, we have that
$\Omega_{F_{M/N}}(C)=\Omega_{F_{M^{\star}}}(C)=\ker\alpha_C$. It remains to show that
$\ker\alpha_C$ does not contain $N$ as a direct summand. Assume that $\ker\alpha_C =
N\oplus N'$. We have the exact sequence
\[0\to N\oplus N'\extto{i} M_C\extto{\alpha_C} C \to 0.\]
Note that $\alpha_C$ is an epimorphism since $M/N$ is a generator for $\mod\L$. Let
$\pi_N\colon N\oplus N'\to N$ be the natural projection. Since every morphism from
$\tau^{-1}N$ to $C$ factors through $\alpha_C$, we have that every morphism from $N\oplus
N'$ to $N$ factors through $i$ (see ~\cite[Chapter IV, Corollary 4.4]{ARS}). In particular $\pi_N$ factors through $i$. But this
implies that $N$ is a direct summand of $M_C$, which is not possible since $M_C$ is in
$\add M/N$. Thus, we have proven that $\ker \alpha_C$ does not contain $N$ as a direct summand. 

Assume now that $l>0$ and let $\{f_1,\ldots,f_l\}$ in $\Hom_\L(\tau^{-1}N,C)$ induce a minimal generating set of the top
of $\underline{\Hom}_{F_{M/N}}(\tau^{-1}N,C)$ as an
$\underline{\End}_{F_{M/N}}(\tau^{-1}N)$-module. We form the following pullback diagram
\[\xymatrix@C=25pt{0\ar[r]& \ker\alpha_C\ar[r]\ar@{=}[d]& Y_1\ar[r]^{g_1}\ar[d]&
\tau^{-1}N\ar[d]^{f_1}\ar[r]& 0 \\
0\ar[r]& \ker\alpha_C\ar[r] & M_C\ar[r]^{\alpha_C}& C \ar[r]& 0 .}\]
Note that the upper sequence is $F_{M/N}$-exact since the lower one is so (recall from Section 1 that $F_{M/N}$-exact sequences are closed under pullbacks). We show that
$g_1$ is a right almost split morphism. Since $f_1$ does not factor through $\alpha_C$, the morphism 
$g_1$ is not a split epimorphism. Let $g\colon X\to \tau^{-1}N$ be a morphism which is
not an isomorphism, with $X$ an indecomposable $\L$-module. Then there exists a morphism $\eta\colon X \to E$ such that
$\varepsilon\comp\eta = g$ (recall that $\varepsilon$ is the minimal right almost split
morphism ending at $\tau^{-1}N$). But
 since $E$ is in $\add M/N$ and the upper sequence of the above diagram is $F_{M/N}$-exact, there exists a morphism $\theta\colon E \to Y_1$ such that
$g_1\comp\theta = \varepsilon$. Then, we have that $g_1\comp\theta\comp\eta = g$, so $g$
factors through $g_1$. Hence, $g_1$ is a right almost split morphism and we have an
isomorphism of short exact sequences

\[\xymatrix@C=25pt{0\ar[r]& Y_1^1\oplus N\ar[r]^-{\left(\begin{smallmatrix} 1& 0 \\ 0 & \zeta
\end{smallmatrix}\right)}\ar[d]^{\wr}& Y_1^1\oplus E
\ar[r]^-{\left(\begin{smallmatrix}0&\varepsilon\end{smallmatrix}\right)}\ar[d]^{\wr}&
\tau^{-1}N\ar[r]\ar@{=}[d]& 0 \\
0\ar[r]& \ker\alpha_C\ar[r]& Y_1\ar[r]^{g_1}& \tau^{-1}N\ar[r]& 0  }\]
and an induced short exact sequence
\[(\mu_1)\colon 0 \to Y_1^1\oplus E \to M_C\oplus \tau^{-1}N\extto{\left(\begin{smallmatrix}\alpha_C & f_1
\end{smallmatrix}\right)}C\to 0\]
which is $F_{M/N}$-exact. If $l=1$, then $\mu_1$ is also $F_{M^{\star}}$-exact and then
$\Omega_{F_{M/N}}(C)\simeq N\oplus Y_1^1$ and $\Omega_{F_{M^{\star}}}(C)\simeq E\oplus
Y_1^1$. Using the same argument as in the case $l=0$, we can show that $Y_1^1$ does
not contain $N$ as a direct summand. If $l>1$, using the short exact sequence $(\mu_1)$, we form the following pullback diagram

\[\xymatrix@C=25pt{0\ar[r] & Y_1^1\oplus E \ar[r]\ar@{=}[d] & Y_2\ar[d]\ar[r]^{g_2} &
\tau^{-1}N\ar[d]^{f_2}\ar[r] & 0\\
0 \ar[r] & Y_1^1\oplus E \ar[r] & M_C\oplus
\tau^{-1}N\ar[r]^-{\left(\begin{smallmatrix}\alpha_C & f_1
\end{smallmatrix}\right)} & C \ar[r] & 0 }\]
As before, we can show that $g_2$ is a right almost split morphism, so we have an
isomorphism of exact sequences
\[\xymatrix@C=25pt{0\ar[r]& Y_2^1\oplus N\ar[r]^-{\left(\begin{smallmatrix} 1& 0 \\ 0 & \zeta
\end{smallmatrix}\right)}\ar[d]^{\wr}& Y_2^1\oplus E
\ar[r]^-{\left(\begin{smallmatrix}0&\varepsilon\end{smallmatrix}\right)}\ar[d]^{\wr}&
\tau^{-1}N\ar[r]\ar@{=}[d]& 0 \\
0\ar[r]& Y_1^1\oplus E \ar[r]& Y_2\ar[r]^{g_2}& \tau^{-1}N\ar[r]& 0.  }\]
Using the isomorphism $Y_2^1\oplus N \simeq Y_1^1\oplus E$ and since $N$ is not a direct
summand of $E$, we get that $Y_2^1\simeq Y_2^2\oplus E$ and $Y_1^1\simeq Y_2^2\oplus N$,
for some $\L$-module $Y_2^2$. So, $\ker\alpha_C\simeq Y_1^1\oplus N\simeq Y_2^2\oplus
N^2$. Moreover, we have an induced short exact sequence
\[(\mu_2)\colon 0 \to Y_2^2\oplus E^2 \to M_C\oplus \tau^{-1}N\oplus \tau^{-1}N
\extto{\left(\begin{smallmatrix}\alpha_C &
f_1 & f_2
\end{smallmatrix}\right)}C\to 0\]
which is $F_{M/N}$-exact. Repeating this process $l$ times, we can show that
$\Omega_{F_{M/N}}(C)=\ker\alpha_C\simeq Y_l^l\oplus N^l$, for some $\L$-module $Y_l^l$
and there is a short exact sequence
\[(\mu_l)\colon 0 \to Y_l^l\oplus E^l \to M_C\oplus \tau^{-1}N\oplus \cdots\oplus\tau^{-1}N
\extto{\left(\begin{smallmatrix}\alpha_C & f_1&\cdots & f_l
\end{smallmatrix}\right)}C\to 0\]
which is $F_{M/N}$-exact. Now since $\{f_1,\ldots,f_l\}$ is a minimal generating set of the
top of $\underline{\Hom}_{F_{M/N}}(\tau^{-1}N,C)$ as an
$\underline{\End}_{F_{M/N}}(\tau^{-1}N)$-module, $\mu_l$ is in addition
$F_{M^{\star}}$-exact. Hence $\Omega_{F_{M^{\star}}}(C)\simeq Y_l^l\oplus E^l$. Moreover,
using the same argument as for the case $l=0$, we can show that $N$ is not a direct
summand of $Y_l^l$ and the proof is now complete.
\end{proof}

For the next lemma, we still keep the notation and assumptions made in the beginning of the section.

\begin{lem}\label{CnotisotoN}
If $C$ is an indecomposable $\L$-module which is not isomorphic to $N$, then
$\pd_{F_{M^\star}}C\leq 1$.
\end{lem}

\begin{proof}
Let $C$ be an indecomposable $\L$-module which is not isomorphic to $N$ and consider the
short exact sequence
\[ 0\to \ker\alpha_C\to M_C \extto{\alpha_C} C\to 0,\]
where $\alpha_C$ is a minimal right $\add M/N$-approximation of $C$. We show that
$\alpha_C$ is in fact a minimal right $\add M$-approximation of $C$. To do this, we need to
show that any morphism from $N$ to $C$ factors through $M_C$. Let $g$ be in
$\Hom_\L(N,C)$.  Recall that $\zeta\colon N\to E$ is the left
almost split morphism starting at $N$. Then, since $C$ is not isomorphic to $N$, there exists a morphism $\theta\colon
E \to C$ such that $\theta\comp\zeta = g$. But then, since $E$ is in $\add M/N$ and
$\alpha_C$ is a right $\add M/N$-approximation, there exists a morphism $\eta\colon E\to
M_C$ such that $\alpha_C\comp\eta = \theta$. So we have that
$g=\alpha_C\comp\eta\comp\zeta$ and we have shown that $\alpha_C$ is a minimal $\add
M$-approximation of $C$. Since $\gldim_{F_M}\L=1$, we then have that
$\Omega_{F_{M/N}}(C)$, which is isomorphic to $\ker\alpha_C$, is in $\add M$. Using now Lemma 
\ref{syzygies} we get that $\Omega_{F_{M^{\star}}}(C)$ is in $\add M/N$ and hence in 
$\add M^{\star}$. So $\pd_{F_{M^{\star}}}C\leq 1$.
\end{proof}

We are now ready to prove the main theorem of this section. The notation and assumptions remain as in the beginning of the section.
\begin{thm}\label{mutation}
$\gldim_{F_{M^{\star}}}\L=1$ if and only if $\Omega_{F_{M/N}}(N)$ is in $\add(M\oplus
M^{\star})$.
\end{thm}

\begin{proof}
In view of Lemma \ref{CnotisotoN}, it is enough to show that $\gldim_{F_{M^{\star}}}\L=1$
if and only if $\pd_{F_{M^{\star}}}(N)\leq 1$. So $\gldim_{F_{M^{\star}}}\L=1$ if and
only if $\Omega_{F_{M^\star}}(N)$ is in $\add M^\star$. By Lemma \ref{syzygies}, and
since $E$ is in $\add M/N$,  we have that $\Omega_{F_{M^\star}}(N)$ is in $\add M^\star$
if and only if $\Omega_{F_{M/N}}(N)$ is in $\add (M\oplus M^\star)$.
\end{proof}

In terms of representation dimension and Auslander generators the above theorem is translated as follows:

\begin{cor}
Let $\L$ be an artin algebra with $\repdim \L = 3$ and $M$ an Auslander generator of $\L$. Let $N$ be an indecomposable, non-projective and non-injective, direct summand of $M$, such that the middle term of the almost split sequence starting at $N$ is in $\add M$. Set $M^{\star}=M/N\oplus \tau^{-1}N$. Then,
 $M^\star$ is an Auslander
generator if and only if $\Omega_{F_{M/N}}(N)$ is in $\add(M\oplus M^{\star})$.
\end{cor}

We complete this section with a result showing that the endomorphism ring of the Auslander generator $M^\star$, that is constructed from the Auslander generator $M$ as above, is derived equivalent to the endomorphism ring of $M$.

\begin{prop}\label{derived}
Let $M=M/N\oplus N$ and $M^\star=M/N\oplus \tau^{-1}N$ be Auslander generators of $\L$, where $N$ is an indecomposable direct summand of $M$ which is not projective or injective. Let also 
\[(\eta)\colon0\to N\to E\to \tau^{-1} N\to 0\]
be the almost split sequence starting at $N$ and assume that $E$ is in $\add M$. Then the endomorphism rings $\End_\L(M)$ and $\End_\L(M^\star)$ are derived equivalent.
\end{prop}

\begin{proof}
The statement is proved by using ~\cite[Proposition 2.10]{L1} and ~\cite[Theorem 2.6]{L1}.  Both $M$ and $M^\star$ are generators-cogenerators and their endomorphism rings have finite global dimension. Moreover the almost split sequence $(\eta)$ satisfies condition (b) of ~\cite[Proposition 2.10]{L1}. That is, $(\eta)$ is a minimal $F^{M^\star}$-injective resolution of $N$ and a minimal $F_M$-projective resolution of $\tau^{-1}N$. Hence, condition (a) of ~\cite[Proposition 2.10]{L1} is satisfied. That is, $\Ext_{F^M}^i(M^\star,M^\star)=(0)$ and $\Ext_{F_{M^\star}}^i(M,M)=(0)$, for $i>0$. Then, by ~\cite[Theorem 2.6]{L1}, this is equivalent to $\Hom_\L(M^\star, M)$ being a cotilting $\End_\L(M^\star)^{\op}$-$\End_\L(M)$-bimodule. This implies that the endomorphism rings $\End_\L(M^\star)$ and $\End_\L(M)$ are derived equivalent (\cite{H}),  which completes the proof.
\end{proof}

Note that Proposition \ref{derived} is true for artin algebras of any representation dimension.

\section{Selfinjective algebras with radical cube zero}

This section deals with selfinjective artin algebras with radical cube zero. If such an algebra is of infinite representation type, it is proven that it has representation dimension equal to three (see ~\cite{A}). We show that for a particular class of Auslander generators of these algebras the rather strict assumptions of the previous section are satisfied. Hence, in this case, we can construct by mutation new Auslander generators from given ones. 

In the rest of the section $\L$ denotes a selfinjective algebra of infinite representation type such that $\rad^3 \L=0$.  We set $A=\L/\Soc \L$. Then $\rad^2 A=0$, so there exists a hereditary algebra $H$ and a functor $F\colon \mod A\to \mod H$ that induces an equivalence between the stable categories $\underline{\mod}A$ and $\underline \mod H$ (see ~\cite[X.2]{ARS}). We denote by $\rrad_\L$ the radical of $\L$ and by $\rrad_A$ the radical of A. We keep the above notation for the rest of this section. We need the following lemma: 

\begin{lem}\label{stequiv}
Let $0\to C\extto{\alpha} B \extto{\beta} S\to 0$ be a short exact sequence in $\mod A$, with $S$ a simple $\L$-module. If $C$ does not contain any simple direct summands, then the sequence $0\to F(C)\to F(B)\to F(S)\to 0$ is also exact.
\end{lem}

\begin{proof}
We need to show that the induced sequences 
\[0\to\rrad_A C \to \rrad_A B \to \rrad_A S \to 0\]
and
\[0\to C/\rrad_A C\to B/\rrad_A B \to S/\rrad_A S\to 0\]
are exact. Since $S$ is simple, we have that $\rrad_A S =0$, $S/\rrad_A S=S$ and the map $\beta$ factors through the natural projection $B\to B/\rrad_A B$, so we have the following pushout diagram

\[\xymatrix@C=25pt{&0\ar[d]&0\ar[d]&&\\
&\rrad_A B\ar[d]^{i}\ar@{=}[r]&\rrad_A B\ar[d]&&\\
0\ar[r]&C\ar[d]\ar[r]^{\alpha}&B\ar[d]\ar[r]^{\beta}&S\ar@{=}[d]\ar[r]&0\\
0\ar[r]&C/\rrad_A B\ar[r]\ar[d]&B/\rrad_A B\ar[r]\ar[d]&S\ar[r]&0\\
&0&0&&
}\]

Since $C$ does not contain any simple direct summands the image of the monomorphism $i\colon \rrad_A B\to C$ is contained in $\rrad_A C$. So $\rrad_A B$ is isomorphic to a submodule of $\rrad_A C$. On the other hand, the monomorphism $\alpha\colon C\to B$ induces a monomorphism $\alpha|_{\rrad_A C}\colon \rrad_A C \to \rrad_A B$. Hence $\rrad_A C =\rrad_ AB$ and consequently $C/\rrad_A B=C/\rrad_A C$ which completes the proof.
\end{proof}

Let $\Sigma$ be a complete slice in the AR-quiver of $H$ that does not contain any projective or injective $H$-modules. For the definition and properties of a complete slice we refer the reader to ~\cite{R} (see also ~\cite{APT}). Consider the $\L$-module $M=\L\oplus M_\P$, where $M_\P=F^{-1}(\Sigma)$. Let $N$ be an indecomposable direct summand of $M_\P$ such that $F(N)$ is a source on $\Sigma$. Then, if 
\[0\to N \extto{\zeta} E \extto{\varepsilon} \tau^{-1}N \to 0\]
is the almost split sequence starting at $N$, the middle term $E$ is in $\add M_\P$ since $\Sigma$ is a complete slice. Set $M^{\star}=M/N\oplus \tau^{-1}N$. Then, with the above assumptions and notation we have the following.

\begin{lem}\label{slicetop}
The top of $M_\P/N$ contains all simple $\L$-modules.
\end{lem}

\begin{proof}
The simple $\L$-modules correspond via the functor $F$ to the simple injective $H$-modules. Hence, it is enough to show that the top of $\Sigma/F(N)$ contains all simple injective $H$-modules. By the definition of a slice, $\Sigma$ is a sincere $H$-module so it contains all simple $H$-modules as composition factors. In particular, it contains all simple injective $H$-modules as composition factors. But a simple injective module can only occur in the top of an indecomposable module, hence the top of $\Sigma$ contains all simple injective $H$-modules. Note that since $N$ is not a simple $\L$-module, the almost split sequence starting at $F(N)$ is:
\[0\to F(N) \extto{F(\zeta)} F(E) \extto{F(\varepsilon)} F(\tau^{-1}N) \to 0.\]
Now, since $F(N)$ is a source on $\Sigma$, $F(E)$ is a direct summand of $\Sigma$. Moreover, since $F(N)$ is not simple, the top of $F(N)$ is a direct summand of the top of $F(E)$ ~\cite{ARS}. Thus, all simple injective $H$-modules are contained in the top of $\Sigma/F(N)$. 
\end{proof}
Now, we assume that the above $M$ is an Auslander generator of $\L$ (we show later in this section that there are Auslander generators of this form) and ask when $M^{\star}$ is again an Auslander generator. The answer is given in the following theorem. We keep the same assumptions for $\L$ and $M$ as above.
\begin{thm}
Assume that $\gldim_{F_{M}}\L=1$.
\begin{enumerate} 
\item[(a)] If $\underline{\Hom}_\L(M_\P/N,N)\neq(0)$, then $\gldim_{F_{M^\star}}\L=1$.
\item[(b)] If $\underline{\Hom}_\L(M_\P/N,N)=(0)$, then $\gldim_{F_{M^\star}}\L=1$ if and only if $\Omega_\L(N)$ is in $\add M_\P$.
\end{enumerate}
\end{thm}

\begin{proof}
Let $p\colon P(N)\to N$ be the projective cover of $N$ and $\alpha\colon M_0\to N$ the minimal right $\add M_\P/N$-approximation of $N$. Then $\Omega_{F_{M/N}}(N)$ can be computed using the following pullback diagram:

\[\xymatrix@C=25pt{0\ar[r]& \Omega_\L (N)\ar@{=}[d]\ar[r]&\Omega_{F_{M/N}}(N)\ar[d]^{\alpha '}\ar[r]^{p'}& M_0\ar[d]^{\alpha}\ar[r]&0\\
0\ar[r]&\Omega_\L (N)\ar[r]& P(N)\ar[r]^{p}&N\ar[r]& 0}
\]

(a) We show that $\Omega_{F_{M/N}}(N)$ is in $\add M$. Then the claim follows from Theorem \ref{mutation}. The first step is to show that we can actually compute $\Omega_{F_{M/N}}(N)$ using a pullback diagram of $A$-modules. This will allow us, using the stable equivalence induced by the functor $F$, to move to the hereditary algebra $H$ and prove the claim using the properties of the slice $\Sigma$. 

Since $N$ is not a simple or a projective $\L$-module,  its Loewy length is two, so $\rrad_\L N = \Soc N$ and we have the following commutative exact diagram: 

\[\xymatrix@C=20pt{&&0\ar[d]&0\ar[d]&\\
0\ar[r]& \Omega_\L (N)\ar@{=}[d]\ar[r]&\rrad_\L P(N)\ar@{^{(}->}[d]\ar[r]& \Soc N\ar@{^{(}->}[d]\ar[r]& 0\\
0\ar[r]& \Omega_\L (N)\ar[r] & P(N) \ar[d]\ar[r]& N \ar[d] \ar[r]& 0\\
&&P(N)/\rrad_\L P(N)\ar[d]\ar^-{\simeq}[r]& N/\Soc N\ar[d]& \\
&&0&0&}\]

\sloppy Since $H$ is hereditary and $F(N)$ is a source on the slice $\Sigma$,
we have that $\Hom_H (F(M_\P/N), F(N))=(0)$. Then  
\[\underline{\Hom}_A(M_\P/N,N)\simeq \underline{\Hom}_H
(F(M_\P/N),F(N)) = (0).\] 
So all morphisms from $M_\P/N$ to $N$ factor through a projective
$A$-module. Since $M_0$ is in $\add M_\P/N$, there exists a projective
$A$-module $Q$ and morphisms $\beta\colon M_0\to Q$ and $\gamma\colon
Q\to N$ such that $\gamma\comp\beta=\alpha$. Since $M_0$ does not
contain any nonzero projective summands we have $\beta (M_0)\subseteq
\rrad_A Q$. Then we have $\gamma\beta(M_0)\subseteq \gamma(\rrad_A
Q)\subseteq\rrad_A N$. Since $N$ is not a simple module we have that
$\rrad_A N = \Soc N$. So the morphism $\alpha$ factors through the
socle of $N$. 

Next observe that $\Omega_{F_{M/N}}(N)$ does not contain any nonzero
projective direct summands. We can see this by looking at the exact
sequence $0\to \Omega_{F_{M/N}}(N)\to P(N)\oplus M_0\extto{(p
  \ \alpha)} N \to 0$. If $\Omega_{F_{M/N}}(N)$ contained a projective
summand $P$, then $P$ would also be injective, hence the monomorphism
$\Omega_{F_{M/N}}(N)\to P(N)\oplus M_0$ would have some split part. In
particular $P$ would be a direct summand of $P(N)$, since $M_0$ is in
$\add M_\P$, which is a contradiction since $P(N)$ is the projective
cover of $N$.  Now, combining these observations with the diagrams
above, we get the following commutative exact diagram: 

\[\xymatrix@C=7pt{0\ar[rr]&&\Omega_\L (N)\ar@{=}[dd]\ar@{=}[dr]\ar[rr]&&\Omega_{F_{M/N}}(N)\ar[dd]\ar[dr]^{\alpha '}\ar[rr]^{p'}&&M_0\ar[dd]\ar[dr]^{\alpha}\ar[rr]&&0&\\
&0\ar[rr]&&\Omega_\L (N)\ar@{=}[dl]\ar[rr]&&\rrad_\L P(N)\ar@{^{(}->}[dl]\ar[rr]^{p|_{\rrad_\L P(N)}}&&\Soc N\ar@{^{(}->}[dl]\ar[rr]&&0\\
0\ar[rr]&&\Omega_\L (N)\ar[rr]&&P(N)\ar[rr]^{p}&&N\ar[rr]&&0&
}\]
Hence we can compute $\Omega_{F_{M/N}}(N)$ using the upper pullback
diagram which is in $\mod A$. By Lemma \ref{slicetop} the morphism
$\alpha$ is an epimorphism, so $\alpha '$ is an epimorphism too. Set
$M_1=\Ker\alpha$. Since $\Soc N$ is not isomorphic to $N$, by Lemma
\ref{CnotisotoN} we get that $M_1$ is in $\add M$. In particular $M_1$
is in $\add M_\P$, since $\alpha$ is minimal (here we use the same
argument as the one we used in order to show that
$\Omega_{F_{M/N}}(N)$ does not contain any nonzero projective
summand). We set for simplicity $p|_{\rrad_\L P(N)}=q$. We have the
following commutative exact diagram of $A$-modules: 

\[\xymatrix@C=20pt{&&0\ar[d]&0\ar[d]&\\
&& M_1\ar[d]^{i}\ar@{=}[r]&M_1\ar[d]&\\
0\ar[r]&\Omega_\L (N)\ar@{=}[d]\ar[r]&\Omega_{F_{M/N}}(N)\ar[d]^{\alpha '}\ar[r]^-{p'}& M_0 \ar[d]^{\alpha}\ar[r]&0\\
0\ar[r]&\Omega_\L (N)\ar[r]&\rrad_\L P(N)\ar[r]^{q}\ar[d]&\Soc N\ar[r]\ar[d]&0\\
&&0&0&}
\] 

The next step is to show that if we apply the functor $F$ to the above
diagram, the resulting commutative diagram will also be exact. To do
this we use Lemma \ref{stequiv}. Note that since $N$ is an
indecomposable not projective $\L$-module, $\Omega_\L (N)$ is also
indecomposable. Moreover $\Omega_\L (N)$ is not a simple $A$-module
(the only $A$-modules $N$ such that $\Omega_\L (N)$ is simple are the
indecomposable projective $A$-modules). Hence, applying Lemma
\ref{stequiv} it is not hard to see that the commutative diagram 

\[\xymatrix@C=20pt{&&0\ar[d]&0\ar[d]&\\
&& F(M_1)\ar[d]^{F(i)}\ar@{=}[r]&F(M_1)\ar[d]&\\
0\ar[r]&F(\Omega_\L (N))\ar@{=}[d]\ar[r]&F(\Omega_{F_{M/N}}(N))\ar[d]^{F(\alpha ')}\ar[r]^-{F(p')}& F(M_0) \ar[d]^{F(\alpha)}\ar[r]&0\\
0\ar[r]&F(\Omega_\L (N))\ar[r]&F(\rrad_\L P(N))\ar[r]^{F(q)}\ar[d]&F(\Soc N)\ar[r]\ar[d]&0\\
&&0&0&}
\] 
is again exact.

The last step is to show that for any indecomposable direct summand $Z$ of $\Omega_{F_{M/N}}(N)$, the compositions 
\[\xymatrix@C=20pt{F(M_1)\ar@{^{(}->}[r]^-{F(i)}& F(\Omega_{F_{M/N}}(N))\ar@{->>}[r]^-{p_Z}& Z\\}\]
and
\[\xymatrix@C=20pt{Z\ar@{^{(}->}[r]^-{i_Z}& F(\Omega_{F_{M/N}}(N))\ar@{->>}[r]^-{F(p')}& F(M_0),\\}\]
where $p_Z$ is the natural projection and $i_Z$ the natural inclusion, are non zero. Since $M_0$ and $M_1$ are in $\add M_\P$ and $F(M_\P)$ is a complete slice, this would mean that $F(\Omega_{F_{M/N}}(N))$ is also in $\add F(M_\P)$. 

So let $Z$ be an indecomposable summand of $\Omega_{F_{M/N}}(N)$, and assume first that the composition $F(p')\comp i_Z$, as above, is zero.  Then, $Z$ is isomorphic to a direct summand of $F(\Omega_\L (N))$. But since $\Omega_\L (N)$ is indecomposable, $F(\Omega_\L (N))$ is also indecomposable (property of the functor $F$), so $Z$ is isomorphic to $F(\Omega_\L (N))$, and the sequence 
\[0\to F(\Omega_\L(N))\to F(\Omega_{F_{M/N}}(N))\extto{F(p')} F(M_0)\]
in the above diagram is split exact. Then $F(\alpha)$ factors through $F(p)$ and consequently $\alpha\colon M_0\to N$ factors through $p\colon P(N)\to N$. But then we have $\underline{\Hom}_\L(M_\P/N,N)=(0)$, which is a contradiction.  

Assume now that the composition $p_Z\comp F(i)$, with $p_Z$ and $F(i)$ as above, is zero. Then, $Z$ is isomorphic to a direct summand of $F(\rrad_\L P(N))$. Note that $F(\rrad_\L P(N))$ is an injective $H$-module, hence $Z$ is an indecomposable injective $H$-module. But then the composition $F(p')\comp i_Z$ is zero, since by assumption $F(M_0)$ does not contain any injective direct summand. As we have already seen, this leads to a contradiction. Hence, for any indecomposable direct summand $Z$ of $F(\Omega_{F_{M/N}}(N))$, both of the compositions $F(p')\comp i_Z$ and $p_Z\comp F(i)$ are nonzero which implies that $F(\Omega_{F_{M/N}}(N))$ is also in $\add F(M_\P)$.

\bigskip

(b) If $\underline{\Hom}_\L(M_\P/N,N)=(0)$, then $\Omega_{F_{M/N}}(N)\simeq M_0\oplus \Omega_\L(N)$. Hence, in view of Theorem \ref{mutation} we have that  $\gldim_{F_{M^\star}}\L=1$ if and only if  $\Omega_\L(N)$ is in $\add M_\P$.
\end{proof}

We apply the above theorem to Auslander generators. Recall from the beginning of the section that $\L$ is a selfinjective algebra of infinite type and $H$ is the hereditary algebra which is stably equivalent to $\L/\Soc\L$ via the functor $F$.

\begin{cor}\label{M^0}
Let $M=\L\oplus F^{-1}(\Sigma)$ be an Auslander generator of $\L$ with $\Sigma$ a complete slice of the AR-quiver of $H$, that does not contain any projective or injective $H$-modules. Then $M^{\star}$ is an Auslander generator if and only if  $\underline{\Hom}_\L(M_\P/N,N)\neq(0)$ or $\Omega_\L(N)$ is in $\add M_\P$.
\end{cor}

\section{Applications and examples}

Let $\L$ be a selfinjective algebra with radical cube zero of infinite representation type. We keep the same assumptions and notation as in the previous section.
A natural question to ask, is whether there exists any Auslander generator $M=\L\oplus M_\P$, with $M_\P$ as before. The answer is that, for these particular algebras, there are infinitely many Auslander generators of this form. Set $$M^0=\L\oplus M_\P^0,$$ where $$M_\P^0=\L/\Soc\L\oplus\L/\Soc^2\L.$$ Then, we know by ~\cite{A}, that $M^0$ is an Auslander generator for $\L$.  Then, the  $\L$-module $$M^i=\L\oplus(\tau^{-1})^i M_\P^0$$ is an Auslander generator by Corollary \ref{DTr}. Moreover, for any $i\geq 1$, we have $(\tau^{-1})^i M_\P^0=F^{-1}(\Sigma_i)$ for some complete slice $\Sigma_i$ in the AR-quiver of $H$ that does not contain any projective or injective $H$-modules.

Starting with  $M^0$, we can construct more Auslander generators by exchanging all simple $\L$-modules, one by one, thus getting a new Auslander generator each time. Of course, we can not apply our result directly to $M^0$, but we can translate $M^0_\P$   by applying the functor $\tau^{-1}$, as we saw above, do the exchange there, and translate the resulting slice back using  $\tau$. Let $M_1$ be the Auslander generator that we have constructed by exchanging all simple $\L$-modules. Since for selfinjective algebras we have $\tau \simeq \Omega ^2\mathcal N$, where $\mathcal N$ is the Nakayama automorphism, we get that $\Omega_\L(M_{1\P})=M^0_\P$. Since 
$\Omega_\L$ induces a stable equivalence on $\mod\L$,  it is not hard to see that we can again exchange all its sources, one by one, thus getting a new Auslander generator each time. Continuing in this way, we can construct by mutation, an infinite set of Auslander generators, which we denote by $\mathcal M$. Note that $\mathcal M$ contains all $M^i=\L\oplus(\tau^{-1})^i M_\P^0$, for $i \geq 0$. 

Moreover, starting with the Auslander generator $$L^0=\L\oplus L_\P^0$$ where $$L_\P^0=\rrad_\L\oplus \rrad^2_\L,$$ the dual construction will give us an infinite set of Auslander generators, which we denote by $\mathcal L$, that contains all $L^i=\L\oplus\tau^i L_\P^0$, for $i\geq 0$. It is shown in  ~\cite{L2} that in the case of the exterior algebra in two variables, the union $\mathcal L\cup \mathcal M$ gives a complete set of all non-isomorphic minimal Auslander generators.

Next, we illustrate the above process with an example.

\begin{example}
Let $Q$ be the quiver 
\[\xymatrix@C=20pt{1\ar@(ul,dl)_c\ar@<0.4ex>[r]^{a}& 2 \ar@<0.4ex>[l]^{b}\ar@(ur,dr)^d}\]
and let $kQ$ be the path algebra of $Q$ over some algebraically closed
field $k$. Set $\L=kQ/I$, where $I$ is the ideal of $kQ$ generated by
$\{c^2-ab,ca,ad,bc,db,d^2-ba\}$. We denote by $P_i$ the indecomposable
projective $\L$-module that corresponds to the vertex $i$, and by
$S_i$, the top of $P_i$, for $i=1,2$. We draw the component of the
AR-quiver of $\L$, that contains the indecomposable projective
$\L$-modules.  
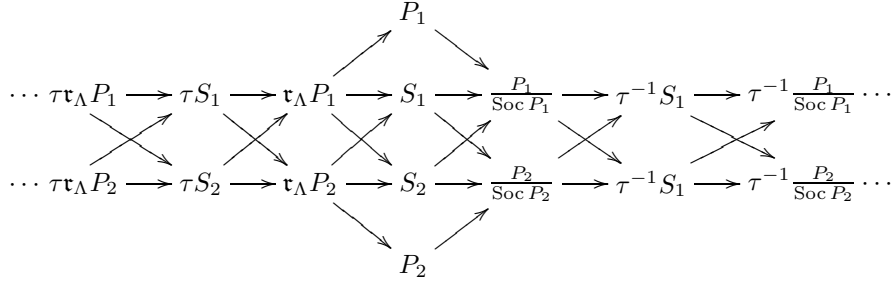
\[\xymatrix@C=0.6cm@R=0.5cm{&&&P_1\ar[dr]&&&\\
\makebox[0.6cm]{$\cdots$}\tau\rrad_\L P_1\ar[r]\ar[dr]&\tau S_1\ar[r]\ar[dr]&\rrad_\L P_1\ar[r]\ar[ur]\ar[dr]&S_1\ar[r]\ar[dr]&\frac{P_1}{\Soc P_1}\ar[r]\ar[dr]&\tau^{-1}S_1\ar[r]\ar[dr]&\tau^{-1}\frac{P_1}{\Soc P_1}\makebox[0.6cm]{$\cdots$}\\
\makebox[0.6cm]{$\cdots$}\tau\rrad_\L P_2\ar[r]\ar[ur]&\tau S_2\ar[r]\ar[ur]&\rrad_\L P_2\ar[r]\ar[dr]\ar[ur]&S_2\ar[r]\ar[ur]&\frac{P_2}{\Soc P_2}\ar[r]\ar[ur]&\tau^{-1}S_1\ar[r]\ar[ur]&\tau^{-1}\frac{P_2}{\Soc P_2}\makebox[0.6cm]{$\cdots$}\\
&&&P_2\ar[ur]&&&
}\]
We begin with the Auslander generator 
\[M^0=P_1\oplus P_2\oplus P_1/\Soc P_1\oplus P_2/\Soc P_2\oplus S_1\oplus S_2.\]
Although $M^0_\P$ can not be obtained from a complete slice of the hereditary algebra that is stable equivalent to $\L/\Soc \L$, 
we can still replace $S_1$ by $\tau^{-1}S_1$ using Corollary \ref{M^0} as follows. We consider the $\L$-module
\begin{multline}M^1=P_1\oplus P_2\oplus \tau^{-1}(P_1/\Soc P_1\oplus P_2/\Soc P_2\oplus S_1\oplus S_2)\notag \\ = P_1\oplus P_2\oplus \tau^{-1}(P_1/\Soc P_1)\oplus \tau^{-1}(P_2/\Soc P_2)\oplus \tau^{-1}S_1\oplus \tau^{-1}S_2\end{multline}
which is again an Auslander generator by Corollary \ref{DTr}. We now apply Corollary \ref{M^0} for $M=M^1$ and $N=\tau^{-1}S_1$. Since $\tau$ induces a stable equivalence on $\mod \L$, we have 
\[\underline{\Hom}_\L(M^1_\P, \tau^{-1}S_1)\simeq \underline{\Hom}_\L(M^0_\P, S_1),\] 
which is nonzero, since the morphism $P_1/\Soc P_1\to S_1$, induced by the projective cover of $S_1$, does not factor through a projective $\L$-module. Hence, the module ${M^1}^{\star}=M^1/\tau^{-1}S_1\oplus \tau^{-2} S_1$ is an Auslander generator. Then, by Corollary \ref{DTr}, we conclude that ${M^0}^{\star}=M^0/S_1\oplus \tau^{-1} S_1$ is an Auslander generator. Now, using exactly the same arguments, we can replace the direct summand $S_2$ of ${M^0}^{\star}$ by $\tau^{-1} S_2$, thus getting the Auslander generator
\[M_1=P_1\oplus P_2\oplus P_1/\Soc P_1\oplus P_2/\Soc P_2\oplus \tau^{-1}S_1\oplus \tau^{-1}S_2.\]
Next, in order to replace the direct summand $P_1/\Soc P_1$ of $M_1$, by $\tau^{-1}(P_1/\Soc P_1)$, and since we can not apply Corollary \ref{M^0} directly to $M_1$, consider the Auslander generator
\begin{multline}M_1^1=P_1\oplus P_2\oplus \tau^{-1}(P_1/\Soc P_1\oplus P_2/\Soc P_2\oplus \tau^{-1}S_1\oplus \tau^{-1}S_2)\notag\\
= P_1\oplus P_2 \oplus  \tau^{-1}(P_1/\Soc P_1)\oplus \tau^{-1}(P_2/\Soc P_2)\oplus \tau^{-2}S_1\oplus \tau^{-2}S_2.\end{multline}
We apply Corollary \ref{M^0} for $M=M_1^1$ and $N=\tau^{-1}(P_1/\Soc P_1)$. Since $\tau$ and $\Omega_\L$ induce a stable equivalence on $\mod\L$, we have  
\begin{multline}\underline{\Hom}_\L({M_1^1}_\P, \tau^{-1}(P_1/\Soc P_1))\simeq \underline{\Hom}_\L({M_1}_\P, P_1/\Soc P_1)\notag \\\simeq \underline{\Hom}_\L(\Omega_\L^{-1}({M_1}_\P), \Omega_\L^{-1}\tau^{-1}(P_1/\Soc P_1))\simeq \underline{\Hom}_\L(M^0_\P, S_1),\end{multline}
which, as we have already seen is nonzero. Hence the module ${M_1^1}^\star=M_1^1/\tau^{-1}(P_1/\Soc P_1)\oplus \tau^{-1}(P_1/\Soc P_1)$ is an Auslander generator which means that $M_1^\star=M_1/(P_1/\Soc P_1)\oplus P_1/\Soc P_1$ is an Auslander generator too. Using the same arguments we can replace the direct summand $P_2/\Soc P_2$ of $M_1^\star$ by $\tau^{-1}(P_2/\Soc P_2)$, thus getting the Auslander generator
\[M_2=P_1\oplus P_2\oplus \tau^{-1}(P_1/\Soc P_1)\oplus \tau^{-1}(P_2/\Soc P_2)\oplus \tau^{-1}S_1\oplus \tau^{-1}S_2.\]
Note that $M_2=M^1$, and that from now on if we continue this process, the nonprojective part of the Auslander generators that we construct are translates of the nonprojective part of those that we have already constructed. So in this example the set $\mathcal M$ consists of the $\L$-modules $M^0, {M^0}^\star, M_1, {M_1}^\star$ and all the modules that are obtained by applying $\tau^i$, for all $i\geq 1$, to the nonprojective part of these modules.

The set $\mathcal L$ is constructed dually, starting from the Auslander generator 
\[L^0=P_1\oplus P_2\oplus \rrad_\L P_1\oplus \rrad_\L P_2 \oplus S_1\oplus S_2.\]
\end{example}

\bigskip 
\bigskip

In general, the mutation process we just described, does not give us all Auslander generators of the form $M=\L\oplus M_\P$, with $M_\P$ coming from a complete slice. The following example, gives an infinite set of nonisomorphic Auslander generators of the form  $M=\L\oplus M_\P$, that can not be obtained in the above way.

\begin{example}
Let $Q$ be the quiver 
\[\xymatrix@C=20pt{1\ar@(ul,dl)_c\ar@<0.4ex>[r]^{a_1}& 2 \ar@<0.4ex>[r]^{a_2}\ar@<0.4ex>[l]^{b_2}& 3\ar@<0.4ex>[l]^{b_3} \ar@<0.4ex>[r]^{a_3}& 4 \ar@<0.4ex>[l]^{b_4}\ar@(ur,dr)^d}\]
and let $kQ$ be the path algebra of $Q$ over some algebraically closed
field $k$. Set $\L=kQ/I$, where $I$ is the ideal of $kQ$ generated by
\[\{c^2-a_1b_2,ca_1,b_2c,a_1a_2,a_2a_3,b_3b_2,b_4b_3,a_2b_3-b_2a_1,a_3b_4-b_3a_2, 
a_3d,db_4,d^2-b_4a_3\}.\] 
We denote by $P_i$ the indecomposable projective $\L$-module that
corresponds to the vertex $i$, and by $S_i$, the top of $P_i$, for
$i=1,2,3,4$. We begin with the Auslander generator:   
\[\L\oplus S_1\oplus S_2 \oplus S_3\oplus S_4\oplus P_1/\Soc P_1\oplus P_2/\Soc  P_2\oplus P_3/\Soc P_3\oplus P_4/\Soc P_4.\]
Exchanging first $S_1$ and then $S_2$ with $\tau^{-1}S_1$ and $\tau^{-1}S_2$ respectively, we get the Auslander generator 
\begin{multline}M=\L\oplus \tau^{-1}S_1\oplus \tau^{-1}S_2 \oplus S_3\oplus S_4\oplus P_1/\Soc P_1\oplus P_2/\Soc P_2\notag\\ \oplus P_3/\Soc P_3\oplus P_4/\Soc P_4.\end{multline}
\sloppy Observe that the $\L$-module $P_1/\Soc P_1$ has now become a source. Hence, instead of exchanging the simple module $S_3$, we can try to exchange the module $P_1/\Soc P_1$. Consider the Auslander generator $\L\oplus \tau^{-1} M_\P$, where $M_\P$ is the non-projective part of $M$. Then $\tau^{-1} M_\P=F^{-1}(\Sigma)$, for some complete slice $\Sigma$ that projective or injective $H$-modules. Set $N=\tau^{-1} (P_1/\Soc P_1)$. Then $F(N)$ is a source on $\Sigma$. Moreover $\underline{\Hom}_\L(\tau^{-1} M_\P/N, N)\neq (0)$, so by Corollary \ref{M^0} we have that  $\L\oplus \tau^{-1} M_\P/\tau^{-1}(P_1/\Soc P_1) \oplus (\tau^{-1})^2 (P_1/\Soc P_1)$ is an Auslander generator. Then, translating back the nonprojective part, we get that 
\begin{multline}M^{\star}=\L \oplus \tau^{-1}S_1\oplus \tau^{-1}S_2 \oplus S_3\oplus S_4\oplus\tau^{-1} (P_1/\Soc P_1) \oplus P_2/\Soc P_2\notag \\ \oplus P_3/\Soc P_3\oplus P_4/\Soc P_4
\end{multline}
is an Auslander generator. Furthermore, for any integer $i$ the $\L$-modules $\L\oplus (\tau^{-1})^i (M^{\star}_\P)$ are Auslander generators. Thus, we have constructed an infinite set of Auslander generators of the form $M=\L\oplus M_\P$, that can not be obtained by the mutation process described above.
\end{example}

\bigskip

The next natural question to ask, is whether all the $\L$-modules $M$ of the form $M=\L\oplus M_\P$ are Auslander generators. The answer is negative as the following example shows.

\begin{example}
Let $Q$ be the quiver 
\[\xymatrix@C=20pt{1\ar@(ul,dl)_c\ar@<0.4ex>[r]^{a_1}& 2 \ar@<0.4ex>[r]^{a_2}\ar@<0.4ex>[l]^{b_2}& 3\ar@<0.4ex>[l]^{b_3} \ar@<0.4ex>[r]^{a_3}&4\ar@<0.4ex>[l]^{b_4} \ar@<0.4ex>[r]^{a_4}& 5 \ar@<0.4ex>[l]^{b_5}\ar@(ur,dr)^d}\]
and let $kQ$ be the path algebra of $Q$ over some algebraically closed
field $k$. Set $\L=kQ/I$, where $I$ is the ideal of $kQ$ generated by
\begin{multline}
\{c^2-a_1b_2,ca_1,b_2c,a_1a_2,a_2a_3,a_3a_4,b_3b_2,b_4b_3,b_5b_4,a_2b_3-b_2a_1,
\notag\\ 
a_3b_4-b_3a_2,a_4b_5-b_4a_3,a_4d,db_5,d^2-b_5a_4\}.  
\end{multline}
We denote by $P_i$ the indecomposable projective $\L$-module that
corresponds to the vertex $i$, and by $S_i$, the top of $P_i$, for
$i=1,2,3,4,5$.  
 Starting with the Auslander generator $$M^0=\L\oplus \L/\Soc\L \oplus \L/\Soc^2\L$$ and iterated applications of Corollary \ref{M^0}, we can construct the Auslander generator:
\begin{multline}M=\L\oplus \tau^{-1}S_1\oplus \tau^{-2}S_2\oplus S_3 \oplus \tau^{-2}S_4\oplus\tau^{-1}S_5\oplus \tau^{-1}(P_1/\Soc P_1)\oplus P_2/\Soc P_2 \notag\\ \oplus \tau^{-1}(P_3/\Soc P_3)\oplus P_4/\Soc P_4\oplus \tau^{-1} (P_5/\Soc P_5).\end{multline}
Consider the Auslander generator $\L\oplus \tau^{-1}M_\P$, where $M_\P$ is the non-projective part of $M$.  Then $\tau^{-1} M_\P=F^{-1}(\Sigma)$, for some complete slice $\Sigma$ that does not contain any projective or injective $H$-modules. Set $N=\tau^{-2}(P_3/\Soc P_3)$. Then $F(N)$ is a source on $\Sigma$. But now, it is not hard to see that $\underline{\Hom}_\L(\tau^{-1} M_\P/N,N)=(0)$ and $\Omega_\L(N)$ is not in $\add (\L\oplus \tau^{-1} M_\P)$ hence, by Corollary \ref{M^0}, the $\L$-module $\L\oplus \tau^{-1} M_\P/N \oplus \tau^{-1} N$, is not an Auslander generator. Hence, the module 
\begin{multline}M^{\star}=\L\oplus \tau^{-1}S_1\oplus \tau^{-2}S_2\oplus S_3 \oplus \tau^{-2}S_4\oplus\tau^{-1}S_5\oplus \tau^{-1}(P_1/\Soc P_1)\oplus P_2/\Soc P_2 \notag\\ \oplus \tau^{-2}(P_3/\Soc P_3)\oplus P_4/\Soc P_4\oplus \tau^{-1} (P_5/\Soc P_5)
\end{multline}
is not an Auslander generator. Furthermore, none of the modules $\L\oplus (\tau^{-1})^i(M^{\star}_\P)$, is an Auslander generator, for any integer $i$. 

\end{example}

To sum up, a complete slice $\Sigma$, in the AR-quiver of $H$, that induces an Auslander generator for the algebra $\L$, can be mutated by replacing a source $N$ of $\Sigma$, by $\tau^{-1}N$ and getting another Auslander generator.  But although there is always some source that you can replace and get again an Auslander generator, not all sources have this property. Note that all of our examples are {\it weakly symmetric} algebras with radical cube zero. These algebras  were classified by Benson in ~\cite{B}. It would be interesting, at least for the tame case, to investigate which mutations of a slice  give an Auslander generator.

\end{document}